\documentclass{amsart}
\DeclareMathSymbol{A}{\mathalpha}{operators}{"41}
\DeclareMathSymbol{B}{\mathalpha}{operators}{"42}
\DeclareMathSymbol{C}{\mathalpha}{operators}{"43}
\DeclareMathSymbol{D}{\mathalpha}{operators}{"44}
\DeclareMathSymbol{E}{\mathalpha}{operators}{"45}
\DeclareMathSymbol{F}{\mathalpha}{operators}{"46}
\DeclareMathSymbol{G}{\mathalpha}{operators}{"47}
\DeclareMathSymbol{H}{\mathalpha}{operators}{"48}
\DeclareMathSymbol{I}{\mathalpha}{operators}{"49}
\DeclareMathSymbol{J}{\mathalpha}{operators}{"4A}
\DeclareMathSymbol{K}{\mathalpha}{operators}{"4B}
\DeclareMathSymbol{L}{\mathalpha}{operators}{"4C}
\DeclareMathSymbol{M}{\mathalpha}{operators}{"4D}
\DeclareMathSymbol{N}{\mathalpha}{operators}{"4E}
\DeclareMathSymbol{O}{\mathalpha}{operators}{"4F}
\DeclareMathSymbol{P}{\mathalpha}{operators}{"50}
\DeclareMathSymbol{Q}{\mathalpha}{operators}{"51}
\DeclareMathSymbol{R}{\mathalpha}{operators}{"52}
\DeclareMathSymbol{S}{\mathalpha}{operators}{"53}
\DeclareMathSymbol{T}{\mathalpha}{operators}{"54}
\DeclareMathSymbol{U}{\mathalpha}{operators}{"55}
\DeclareMathSymbol{V}{\mathalpha}{operators}{"56}
\DeclareMathSymbol{W}{\mathalpha}{operators}{"57}
\DeclareMathSymbol{X}{\mathalpha}{operators}{"58}
\DeclareMathSymbol{Y}{\mathalpha}{operators}{"59}
\DeclareMathSymbol{Z}{\mathalpha}{operators}{"5A}
\usepackage[final]{microtype}
\usepackage[a4paper,text={160mm,254mm},centering]{geometry}
\usepackage{hyperref,amsthm,amssymb,tikz-cd,mathtools}
\usepackage[cal=euler,scr=rsfs]{mathalpha}

\theoremstyle{remark}
\newtheorem{remark}{Remark}[section]
\theoremstyle{plain}
\newtheorem{theorem}[remark]{Theorem}
\newtheorem{lemma}[remark]{Lemma}
\theoremstyle{definition}
\newtheorem{definition}[remark]{Definition}

\numberwithin{equation}{remark}

\newtheoremstyle{void} {.5\baselineskip plus .2\baselineskip minus .2\baselineskip} {.5\baselineskip plus .2\baselineskip minus .2\baselineskip} {\normalfont} {} {\bfseries} {.}
{5pt plus 1pt minus 1pt} 
{\thmname{#1 }\thmnumber{#2}\thmnote{.~{#3}}}
\theoremstyle{void}

\newtheorem{situation}[remark]{}

\AtBeginDocument{
  \renewcommand{\setminus}{\mathbin{\rule[0.2em]{0.67em}{0.12em}}}%
  
  \renewcommand{\leq}{\leqslant}
  \renewcommand{\geq}{\geqslant}
  \renewcommand{\subset}{\subseteq}
  
}

\title{Weil conjectures and affine hypersurfaces}

\subjclass{Primary: 11G25;  Secondary: 14G15}

\author[Dingxin Zhang]{Dingxin Zhang}

\address{Center for Mathematics and Interdisciplinary Sciences, Fudan
  University, and Shanghai Institute for Mathematics and
  Interdisciplinary Sciences (SIMIS), Shanghai 200433, China}
\email{dingxinzhang@fudan.edu.cn}
\begin{document}

\begin{abstract}
We give yet another proof of the Riemann hypothesis for smooth proper
varieties over a finite field, by reducing to the case of a
hypersurface~\cite{katz_riemann-hypothesis-for-curves-and-hypersurfaces}
via deformation. The main tool is Artin's vanishing theorem together
with a few basic facts about perverse sheaves.
\end{abstract}

\maketitle

\section{Introduction}
Throughout this paper, we work over a finite base field
\(\mathbf{F}_q\).  Our notation is as follows: capital Latin letters
with a subscript \(0\), such as \(X_0, S_0\), denote separated schemes
of finite type defined over \(\mathbf{F}_q\).  Removing the subscript
(e.g., \(X, S\)) indicates the base change of these schemes to a fixed
algebraic closure \(\overline{\mathbf{F}}_q\).

We fix a prime number \(\ell\) which is invertible in
\(\mathbf{F}_q\), and throughout, we consider only
\(\overline{\mathbf{Q}}_{\ell}\)-sheaves.  Script letters with
subscript \(0\), like \(\mathcal{F}_0, \mathcal{G}_0\), refer to
sheaves on \(X_0, S_0\). Omitting the subscript denotes their pullback
to the corresponding scheme over \(\overline{\mathbf{F}}_q\).

\medskip
Let \(X_0\) be a separated scheme of finite type over \(\mathbf{F}_q\).
We define the zeta function of \(X_0\) as
\[
Z(X_0/\mathbf{F}_q, t) = \exp\left \{ \sum_{e \geq 1} \# X_0(\mathbf{F}_{q^e})\frac{t^e}{e} \right \},
\]
which is an element of \(1 + t\mathbf{Z}[\![t]\!]\).
Grothendieck's trace formula expresses the zeta function in terms of the action of the geometric Frobenius on the \(\ell\)-adic cohomology of \(X\):
\[
Z(X_0/\mathbf{F}_q, t) = \prod_{i=0}^{2\dim X_0}
\det\left(1 - tF, \mathrm{H}^i_c(X; \overline{\mathbf{Q}}_\ell)\right)^{(-1)^{i+1}}.
\]
Recall that \(F\), the \emph{geometric} Frobenius, is the inverse of
the arithmetic Frobenius automorphism \(a \mapsto a^q\) in
\(\mathrm{Gal}(\overline{\mathbf{F}}_q/\mathbf{F}_q)\), acting on the
\(\ell\)-adic cohomology groups by transport of structure.

Our purpose is to provide an alternative proof of the following
celebrated theorem of Deligne, known as the \emph{Riemann Hypothesis}
for smooth proper varieties defined over a finite field:

\begin{theorem}[Deligne~\cite{deligne_weil1, deligne_weil2}]
\label{theorem:rh}
Assume \(X_0\) is smooth and proper over \(\mathbf{F}_q\).
Suppose \(\alpha\) is an eigenvalue of the geometric Frobenius \(F\)
acting on
\(\mathrm{H}^i_{c}(X;\overline{\mathbf{Q}}_{\ell})=\mathrm{H}^i(X;\overline{\mathbf{Q}}_{\ell})\).
Then for any isomorphism
\(\iota\colon \overline{\mathbf{Q}}_{\ell} \to \mathbf{C}\), we have
\(|\iota(\alpha)|=q^{\frac{i}{2}}\).
\end{theorem}

In terms of the definition below, for a smooth proper variety
\(X_0\) over \(\mathbf{F}_q\), the Riemann Hypothesis asserts that
every eigenvalue of the geometric Frobenius acting on
\(\mathrm{H}^i(X;\overline{\mathbf{Q}}_{\ell})\) has \(q\)-weight
exactly equal to \(i\).

\begin{definition}[Weight]
Let \(\iota\colon \overline{\mathbf{Q}}_{\ell} \to \mathbf{C}\) be an
isomorphism of fields.  Let \(q > 0\) be a real number, and let
\(\alpha \in \overline{\mathbf{Q}}_{\ell}\).  We say that \(\alpha\)
has \(q\)-\emph{weight} (or simply \emph{weight} when the number \(q\)
is understood from the context) \(\leq w\) (with respect to
the isomorphism \(\iota\)), if
\[
|\iota(\alpha)| \leq q^{w/2}.
\]
We say \(\alpha\) has \(q\)-weight equal to \(w\), if
\(|\iota(\alpha)| = q^{w/2}\).
\end{definition}

Katz~\cite{katz_riemann-hypothesis-for-curves-and-hypersurfaces}
provides an (arguably) elementary and accessible proof of the Riemann
Hypothesis for smooth projective hypersurfaces in a projective space.
We shall reduce Theorem~\ref{theorem:rh}
to this case.

\begin{theorem}
\label{theorem:katz}
Let \(X_{0} \subset \mathbf{P}^{n}\) be a smooth hypersurface.  Then
the Riemann Hypothesis holds for \(X_{0}\).
\end{theorem}

Katz's proof of Theorem~\ref{theorem:katz} is in effect a
specialization argument.  Building on Deligne's interpretation of
Rankin's method, Katz showed that to establish the Riemann Hypothesis
for \emph{all} smooth hypersurfaces of degree \(d\) in
\(\mathbf{P}^n\), it is enough to find just \emph{one} smooth
hypersurface of degree \(d\) in \(\mathbf{P}^n\) for which the Riemann
Hypothesis holds.  Depending on whether the degree \(d\) is coprime to
the characteristic of the base field, we can construct the following
examples:
\begin{itemize}
\item If \(d\) is coprime to the characteristic of \(\mathbf{F}_q\) then
Weil~\cite{weil_equations-in-finite-fields} proved, using only basic
properties of Gauss sums, that the so-called \emph{diagonal hypersurfaces}
\[
a_{0} T_{0}^{d} + \cdots + a_{n} T_{n}^{d} = 0, \qquad (a_{0}, \ldots,
a_{n} \in \mathbf{F}_q^{\ast})
\]
satisfy the Riemann Hypothesis.
\item If \(d\) is divisible by the characteristic, then Katz used
properties of Gauss sums to verify the Riemann Hypothesis for the
so-called \emph{Gabber hypersurface}:
\[
T_{0}^{d} + T_{0} T_{1}^{d-1} + \cdots + T_{n-1} T_{n}^{d-1} = 0.
\]
\end{itemize}

In this paper, we use a degeneration argument to reduce the
general case of the Riemann Hypothesis for smooth proper varieties
to the hypersurface case covered by Theorem~\ref{theorem:katz}.
Earlier, Scholl~\cite{scholl_hypersurfaces-and-weil-conjectures}
developed a reduction method relying on alterations and the weight
spectral sequence of Steenbrink--Rapoport--Zink.  In contrast, the
engine of our approach is Artin's vanishing
theorem~\cite[Exposé~XIV,~3.1]{sga4}, used together with a simple
perverse degeneration lemma for families over a curve
(\cite[Proposition~9]{katz_perverse-origin},
\cite[Lemma~3.1]{wan-zhang_betti-number-bounds-for-varieties-and-exponential-sums0}).
We recall only the basic facts about the perverse t-structure that are
needed for this argument.

In fact, we will prove a slightly more general assertion, which is a
direct consequence of the main theorem of Weil
II~\cite{deligne_weil2}:

\begin{theorem}\label{theorem:wrh}
Let \(X_{0}\) be a separated scheme of finite type over
\(\mathbf{F}_{q}\).  Then for every integer \(i\), all eigenvalues of Frobenius acting on
\(\mathrm{H}^{i}_{c}(X;\overline{\mathbf{Q}}_{\ell})\) have
\(q\)-weight \(\leq i\).
\end{theorem}

\begin{proof}[Proof that Theorem~\ref{theorem:wrh} \(\Rightarrow\) Theorem~\ref{theorem:rh}]
This is a consequence of the Poincaré duality.
Suppose \(X_{0}\) is of pure dimension \(d\).
Let \(0\leq i\leq 2d\).  Then by Theorem~\ref{theorem:wrh}, any
Frobenius eigenvalue \(\alpha\) of
\(\mathrm{H}^{i}(X;\overline{\mathbf{Q}}_{\ell})\) satisfies the
inequality
\[
|\iota(\alpha)| \leq q^{\frac{i}{2}}.
\]
On the other hand, Poincaré
duality gives an isomorphism
\[
\mathrm{H}^i(X; \overline{\mathbf{Q}}_\ell) \simeq \mathrm{H}^{2d-i}(X; \overline{\mathbf{Q}}_\ell)^{\vee} \otimes \overline{\mathbf{Q}}_\ell(-d).
\]
Therefore, \(q^{d}/\alpha\) is also a Frobenius eigenvalue of \(\mathrm{H}^{2d-i}(X;\overline{\mathbf{Q}}_{\ell})\).
Hence
\[
\left| \iota\left( \frac{q^{d}}{\alpha} \right) \right| \leq q^{d-\frac{i}{2}},
\]
or equivalently \(|\iota(\alpha)| \geq q^{\frac{i}{2}}\).
This completes the proof.
\end{proof}

The paper is organized as follows.  Section~\ref{sec:perv} recalls the
few facts about perverse sheaves that we need.  In
Section~\ref{sec:trivial}, we explain the lower semicontinuity of
weights.  Finally, in Section~\ref{sec:proof}, we assemble these
ingredients to prove Theorem~\ref{theorem:wrh}.

\medskip\noindent%
\textbf{Acknowledgments.}
This proof arose from teaching a course on the Weil conjectures in
Fall 2025.  At first, I planned to cover Scholl's reduction, but I
noticed that the perverse degeneration
lemma~\ref{lemma:perverse-degeneration-lemma} could be used to give an
arguably less technical proof.  The other key input, i.e., Artin's
vanishing theorem, was needed in
\cite{deligne_weil1,scholl_hypersurfaces-and-weil-conjectures} anyway.
I am grateful to the students in my class for their questions and
feedback.

I worked out the details while at the Tianyuan Mathematical Research
Center, during the event ``Exponential sums: Theory, Computation, and
Applications''
organized by Daqing Wan and Ping Xi.  I thank the organizers for
inviting me, and the Center for a wonderful working environment.

I would like to thank Shizhang Li, Daqing Wan, and an anonymous
referee for carefully reading the manuscript and for their
suggestions that improved the exposition of this paper.  I am also
grateful to Tony Scholl for sharing the historical background and
motivation behind his
paper~\cite{scholl_hypersurfaces-and-weil-conjectures}.

\section{A few facts about the perverse t-structure}
\label{sec:perv}

Let \(k\) be either an algebraically closed field or a finite field,
and let \(X\) be a variety over \(k\), i.e., a separated scheme of
finite type over \(k\).  We write \(\mathrm{D}^b_c(X)\) for the
bounded derived category of constructible complexes of
\(\overline{\mathbf{Q}}_{\ell}\)-sheaves on \(X\), where \(\ell\) is a
prime number invertible in \(k\).  This category admits the middle
perversity t-structure
\[
({}^{\mathrm{p}}\mathrm{D}^{\leq0}_{c}(X), {}^{\mathrm{p}}\mathrm{D}_{c}^{\geq0}(X)).
\]
We shall not reproduce the definition here.  See~\cite[\S\S1.3, 1.4,
and 2.2.12]{bbd} or \cite[Lemma-Definition
III~1.1]{kiehl-weissauer_weil-conjectures-perverse-sheaves}.


Objects in \({}^{\mathrm{p}}\mathrm{D}^{\leq0}_{c}(X)\)
(resp.~\({}^{\mathrm{p}}\mathrm{D}^{\geq0}_{c}(X)\)) are called
\emph{perverse connective} (resp.~\emph{perverse coconnective}).
Objects in the heart
\[
{}^{\mathrm{p}}\mathrm{D}^{\leq0}_{c}(X) \cap  {}^{\mathrm{p}}\mathrm{D}_{c}^{\geq0}(X)
\]
of this t-structure are called \emph{perverse sheaves}.

It is very easy to describe the perverse t-structure on a curve.  For
our purposes, we only need the following description, see
\cite[p.~136]{kiehl-weissauer_weil-conjectures-perverse-sheaves}.

\begin{definition}\label{lemma:coconn-curve}
Let \(B\) be a smooth, geometrically connected curve over an
algebraically closed field \(k\).  Let
\(\mathcal{F} \in \mathrm{D}^{b}_{c}(B)\).  Then
\(\mathcal{F} \in {}^{\mathrm{p}}\mathrm{D}^{\geq0}_{c}(B)\) if and only if the
following two conditions hold:
\begin{enumerate}
\item \(\mathrm{H}^{e}(\mathcal{F}_{\overline{\eta}}) = 0\) for all
\(e \leq -2\), where \(\overline{\eta}\) is a geometric generic point
of \(B\), and
\item \(\mathrm{H}^{e}(R\iota_{s}^{!} \mathcal{F}) = 0\) for every
\(e\leq -1\), where \(\iota_{s}\colon \{s\} \to B\) is a closed
immersion of a closed point.
\end{enumerate}
\end{definition}


\begin{lemma}[Perverse degeneration lemma, cf.~{\cite[Proposition~9]{katz_perverse-origin}, \cite[Lemma~3.1]{wan-zhang_betti-number-bounds-for-varieties-and-exponential-sums0}}]
\label{lemma:perverse-degeneration-lemma}
Let \(B\) be a smooth connected curve over an algebraically closed
field.  Let \(\mathcal{F}\in{}^{\mathrm{p}}\mathrm{D}^{\geq0}_c(B)\)
such that \(j^{\ast}\mathcal{F}\) is lisse for a nonempty
open immersion \(j\colon U \to B\).
Then, for every closed point \(s\in B\), there is a natural injective
map
\[
\mathrm{H}^{-1}(\mathcal{F}_s)
\hookrightarrow
\bigl(\mathrm{R}^0j_*\mathcal{H}^{-1}(j^*\mathcal{F})\bigr)_s.
\]
\end{lemma}

\begin{proof}
Let \(\iota\colon \{s\} \to B\) be the closed immersion and
\(j\colon B\setminus\{s\}\to B\) the open immersion.  After replacing
\(B\setminus \{s\}\) by an open dense subscheme, we can assume that
\(\mathcal{H}^{e}(j^{\ast}\mathcal{F})\) are lisse sheaves on
\(B\setminus\{s\}\) for all \(e \in \mathbf{Z}\).

Consider the distinguished triangle
\[
\iota_{\ast}\mathrm{R}\iota^{!}\mathcal{F} \to \mathcal{F} \to \mathrm{R}j_{\ast} j^{\ast}\mathcal{F} \to \iota_{\ast}\mathrm{R}\iota^{!}\mathcal{F}[1].
\]
Applying the exact functor \(\iota^{\ast}\), we
obtain a distinguished triangle in \(\mathrm{D}^b_c(\{s\})\)
\[
\mathrm{R}\iota^{!}\mathcal{F} \to \iota^{\ast}\mathcal{F} \to \iota^{\ast}\mathrm{R}j_{\ast} j^{\ast}\mathcal{F} \to \mathrm{R}\iota^{!}\mathcal{F}[1].
\]
Passing to cohomology, this induces a long exact sequence:
\[
\cdots \to \mathrm{H}^{-1}(\mathrm{R}\iota^{!}\mathcal{F}) \to \mathrm{H}^{-1}(\iota^{\ast}\mathcal{F}) \to \mathrm{H}^{-1}(\iota^{\ast}\mathrm{R}j_{\ast} j^{\ast}\mathcal{F}) \to \mathrm{H}^0(\mathrm{R}\iota^{!}\mathcal{F}) \to \cdots
\]
Thanks to Definition~\ref{lemma:coconn-curve}(2), we have
\(\mathrm{H}^{e}(\mathrm{R}\iota^{!}\mathcal{F}) = 0\) for all \(e \leq -1\).
In particular, \(\mathrm{H}^{-1}(\mathrm{R}\iota^{!}\mathcal{F}) = 0\).  Thus,
from the long exact sequence above, we obtain an injective map
\begin{equation}\label{equation:injective-map}
\mathrm{H}^{-1}(\iota^{\ast}\mathcal{F}) \hookrightarrow \mathrm{H}^{-1}(\iota^{\ast}\mathrm{R}j_{\ast} j^{\ast}\mathcal{F}).
\end{equation}
Since \(\iota^{\ast}\) is exact,
\(\mathrm{H}^{-1}(\iota^{\ast}\mathrm{R}j_{\ast} j^{\ast}\mathcal{F})\) coincides with the
stalk at \(s\) of the sheaf \(\mathrm{R}^{-1}j_{\ast} j^{\ast}\mathcal{F}\).

We now analyze \(\mathrm{R}^{-1}j_{\ast} j^{\ast}\mathcal{F}\) via the spectral sequence
\[
E_2^{a,b} = \mathrm{R}^a j_{\ast}\, \mathcal{H}^b(j^{\ast}\mathcal{F}) \implies \mathrm{R}^{a+b}j_{\ast} j^{\ast}\mathcal{F}.
\]

Since \(\mathcal{H}^{b}(j^{\ast}\mathcal{F})\) is a lisse sheaf on
\(B\setminus\{s\}\), and its stalk at a geometric generic point
\(\overline{\eta}\) of \(B\setminus\{s\}\) is zero when \(b\leq -2\)
by Definition~\ref{lemma:coconn-curve}(1), we have
\(\mathcal{H}^b(j^*\mathcal{F}) = 0\) when \(b \leq -2\).  On the
other hand, for any lisse sheaf \(\mathcal{E}\) on
\(B\setminus\{s\}\), one has \(\mathrm{R}^a j_* \mathcal{E} = 0\)
unless \(a \geq 0\).  Since \(\mathcal{H}^{b}(j^{\ast}\mathcal{F})\)
is lisse for any \(b\), we see that \(E_2^{a, b} = 0\) whenever
\(a + b \leq -2\), and the only nontrivial contribution to
\(\mathrm{R}^{-1}j_* j^* \mathcal{F}\) arises from \(a = 0\),
\(b = -1\). That is,
\[
\mathrm{R}^{-1}j_{\ast}j^{\ast}\mathcal{F} = \mathrm{R}^0 j_{\ast} \mathcal{H}^{-1}(j^{\ast}\mathcal{F}).
\]
Therefore,
\begin{equation}\label{equation:stalk-of-direct-image}
\mathrm{H}^{-1}\bigl(\iota^{\ast}\mathrm{R}j_{\ast} j^{\ast}\mathcal{F}\bigr) = \left(\mathrm{R}^0 j_{\ast} \mathcal{H}^{-1}(j^{\ast}\mathcal{F})\right)_s.
\end{equation}
The proof is complete by combining the injective map \eqref{equation:injective-map} with \eqref{equation:stalk-of-direct-image}.
\end{proof}

Below, we summarize several facts about the perverse t-structure that
will be used later.  To keep the exposition clear, we will not include
detailed proofs here, but will provide references for each result.
The proofs are conceptually simple, with the exception of Artin's
vanishing theorem (Theorem~\ref{theorem:relative-affine-lefschetz}),
whose proof is substantially more subtle.

\begin{lemma}\label{lemma:perversity-local-system-on-smooth-variety}
Let $X$ be a variety of dimension $\leq n$.  Then for any plain sheaf
\(\mathcal{F}\) on \(X\), $\mathcal{F}[n]$ is perverse connective.  In
particular, if \(X\) is smooth of pure dimension \(n\), and
\(\mathcal{F}\) is lisse, then \(\mathcal{F}[n]\) is a
perverse sheaf.
\end{lemma}

\begin{proof}
The first assertion follows straightforwardly from the definition of a
perverse connective complex (see \cite[Lemma-Definition
III~1.1]{kiehl-weissauer_weil-conjectures-perverse-sheaves}).  For the
second assertion, we use the fact that for a lisse sheaf
\(\mathcal{F}\) on a smooth variety, we have
\(\mathbf{D}_X(\mathcal{F}[n]) = \mathcal{F}^{\vee}[n](n)\).
\end{proof}

\begin{lemma}\label{lemma:local-equation-and-perverse-coconnective}
Let \(X\) be a variety and let \(\mathcal{F} \in {}^{\mathrm{p}}\mathrm{D}^{\geq0}_{c}(X)\) be perverse coconnective.
Suppose \(V\)
is a closed subscheme of \(X\) locally cut out by \(r\) regular
functions.  Then \(\mathcal{F}[-r]|_V\) is perverse
coconnective.
\end{lemma}

\begin{proof}
See \cite[Lemma~2.1.1]{wan-zhang_betti-number-bounds-for-varieties-and-exponential-sums0}.
\end{proof}

It follows immediately that if \(X\) is smooth of pure dimension \(n\),
\(V \subset X\) is locally cut out by \(r\) equations, and
\(\dim V = n - r\), then \(\overline{\mathbf{Q}}_{\ell,V}[n-r]\) is a
perverse sheaf on \(V\).

\medskip%
We next state Artin's vanishing theorem which underlies the various
cohomological Lefschetz hyperplane theorems in algebraic geometry, and
is also a key tool of this paper.

\begin{theorem}[Artin's vanishing theorem]
\label{theorem:relative-affine-lefschetz}
Let \(\pi\colon X \to S\) be an affine morphism of varieties. Let
\(\mathcal{F} \in {}^{\mathrm{p}}\mathrm{D}_c^{\geq0}(X)\) be perverse coconnective on
\(X\).  Then \(\mathrm{R}\pi_{!}\mathcal{F}\) is perverse coconnective
on \(S\).
\end{theorem}

\begin{proof}
See \cite[Exposé~XIV,~3.1]{sga4} or \cite[Théorème~4.1.1]{bbd}.
\end{proof}

Here is another useful result not easily found in standard references
on étale cohomology or perverse sheaves.  This statement is a
consequence of Deligne's perverse weak Lefschetz
theorem~\cite[Appendix]{katz_affine-cohomological-transforms}, which
in turn is based on Artin's vanishing theorem
together with Deligne's generic base change
theorem~\cite{deligne_finitude}.

\begin{lemma}[Perverse weak Lefschetz theorem]
\label{lemma:gysin}
Let \(X\) be an algebraic variety over an algebraically closed field
\(k\), and let \(X \to \mathbf{P}^{N}\) be a quasi-finite
morphism.  Suppose
\(\mathcal{F} \in {}^{\mathrm{p}}\mathrm{D}^{\leq0}_{c}(X)\) is
perverse connective.  Then, for a general hyperplane
\(B \subset \mathbf{P}^{N}\), the Gysin map
\[
\mathrm{H}^{i-2}_{c}\left(X \times_{\mathbf{P}^{N}} B; \mathcal{F}|_{X\times_{\mathbf{P}^{N}}B}(-1)\right) \longrightarrow \mathrm{H}^{i}_{c}(X; \mathcal{F})
\]
is surjective for \(i = 1\), and is an isomorphism for all \(i \geq 2\).
\end{lemma}

\begin{proof}[Sketch of proof]
Set \(\mathcal{G} = \mathbf{D}\mathcal{F}\).  Then
\(\mathcal{G} \in {}^{\mathrm{p}}\mathrm{D}^{\geq0}_c(X)\).  By
Deligne's perverse weak Lefschetz
theorem~\cite[Corollary~A.5]{katz_affine-cohomological-transforms}
(while the statement was for perverse sheaves, the proof only uses
that the complex is coconnective in order to apply Artin's vanishing
theorem), for a general hyperplane \(B\), the restriction map
\begin{equation}
\label{eq:deligne-lef}
\mathrm{H}^i(X;\mathcal{G}) \to \mathrm{H}^i(X\times_{\mathbf{P}^N}B;\mathcal{G}|_{B})
\end{equation}
is injective if \(i = -1\), and bijective if \(i \leq -2\).  Using
Deligne's generic base change
(\cite[Corollaire~2.9]{deligne_finitude}), one can show that for \(B\)
general, the morphism \(\iota\colon X \times_{\mathbf{P}^N} B \to X\)
also satisfies
\begin{equation}
\label{eq:gen-purity}
\mathrm{R}\iota^{!}\mathcal{F} \simeq \iota^{\ast}\mathcal{F}[-2](-1)
\end{equation}
(see, e.g., \cite[Lemma~3.2]{wan-zhang_colevel} for details).
The lemma follows from applying the biduality theorem
to \eqref{eq:deligne-lef}
and then using \eqref{eq:gen-purity}.
\end{proof}

\section{Semicontinuity of weights}
\label{sec:trivial}

Before proving Theorem~\ref{theorem:wrh},
we establish the semicontinuity of weights (Lemma~\ref{lemma:semicont}),
a cohomological enhancement of the ``trivial estimates''
(Lemma~\ref{lemma:simple}).
This result is standard
(see \cite[Lemme~1.8.1]{deligne_weil2},
\cite[Lemma~I~2.5]{kiehl-weissauer_weil-conjectures-perverse-sheaves}),
but we include a proof for clarity and completeness.
We first recall the action of geometric Frobenius on
\(\overline{\mathbf{Q}}_{\ell}\)-sheaves and the L-function formalism.

\begin{definition}
Let \(X_{0}\) be a separated scheme of finite type over
\(\mathbf{F}_{q}\).  Let \(\mathcal{F}_{0}\) be a
\(\overline{\mathbf{Q}}_{\ell}\)-sheaf on \(X_{0}\).  For any point
\(x \in X_{0}(\mathbf{F}_{q^{e}})\), we get a morphism
\[
\overline{x}\colon \operatorname{Spec}\overline{\mathbf{F}}_{q} \to \operatorname{Spec}\mathbf{F}_{q^e} \xrightarrow{x} X_{0}.
\]
Then the \emph{inverse} of the Frobenius automorphism of
\(\mathrm{Gal}(\overline{\mathbf{F}}_{q}/\mathbf{F}_{q^{e}})\), called
the \emph{geometric Frobenius}, acts
on the geometric stalk \(\mathcal{F}_{0\overline{x}}\).  We denote this action by
\[
F_{x} \colon \mathcal{F}_{0 \overline{x}} \to \mathcal{F}_{0 \overline{x}}.
\]
\end{definition}

\begin{definition}
Let \(X_0\) be a separated scheme of finite type over \(\mathbf{F}_q\).
Let \(\mathcal{F}_0\) be a \(\overline{\mathbf{Q}}_{\ell}\)-sheaf on \(X_0\).
Then the L-function of \(\mathcal{F}_{0}\) is
\[
L(\mathcal{F}_0,t) =
\exp\left( \sum_{e\geq1}
\left(\sum_{x\in X_0(\mathbf{F}_{q^e})} \operatorname{Tr}(F_x,\mathcal{F}_{0\overline{x}}) \right) \frac{t^e}{e} \right).
\]
By Grothendieck's trace formula, we have
\[
L(\mathcal{F}_0,t) = \prod_i \det(1-tF,\mathrm{H}^i_c(X;\mathcal{F}))^{(-1)^{i-1}}.
\]
\end{definition}

\begin{definition}[Pointwise weight of a sheaf]
In the situation above, we say \(\mathcal{F}_{0}\) is of
\emph{pointwise weight} \(\leq w\), if for any \(e\geq 1\), any point
\(x \in X_{0}(\mathbf{F}_{q^{e}})\), the eigenvalues of
\(F_{x}\colon \mathcal{F}_{0\overline{x}} \to
\mathcal{F}_{0\overline{x}}\) have \(q^{e}\)-weight \(\leq w\).
\end{definition}

\begin{remark}
Suppose \(x\) is an \(\mathbf{F}_{q^e}\)-valued point and \(y\) is an
\(\mathbf{F}_{q^{e^{\prime}}}\)-valued point lying over \(x\), that is,
suppose we have a commutative diagram
\[
\begin{tikzcd}
\operatorname{Spec}\overline{\mathbf{F}}_{q} \ar[r] &
\operatorname{Spec}\mathbf{F}_{q^{e^{\prime}}} \ar[r] \ar[rr,"y",bend left=30] & \operatorname{Spec}\mathbf{F}_{q^e} \ar[r,"x"] & X_{0}
\end{tikzcd}.
\]
Then \(e\) divides \(e^{\prime}\), and, as
\(\overline{\mathbf{Q}}_{\ell}\)-linear maps on
\(\mathcal{F}_{0 \overline{x}} = \mathcal{F}_{0 \overline{y}}\), we
have \(F_x^{e^{\prime}/e} = F_y\).  For this reason,
eigenvalues of \(F_x\) all have \(q^e\)-weight \(\leq w\) if and only
if eigenvalues of \(F_y\) have \(q^{e^{\prime}}\)-weight \(\leq w\).
\end{remark}

\begin{lemma}[Trivial bounds~{\cite[Proposition~1.4.6]{deligne_weil2}}]
\label{lemma:simple}
Let \(X_0\) be a separated scheme of finite type over
\(\mathbf{F}_{q}\) of dimension \(\leq d\).  Suppose
\(\mathcal{F}_{0}\) is a \(\overline{\mathbf{Q}}_{\ell}\)-sheaf on
\(X_{0}\) of pointwise weight \(\leq w\).  Then for any isomorphism
\(\iota\colon\overline{\mathbf{Q}}_{\ell}\xrightarrow{\sim}\mathbf{C}\),
the L-function \(\iota L(\mathcal{F}_0,t)\) is convergent for all
\(|t| < q^{-w/2-d}\), and has neither zeros nor poles in this region.
\end{lemma}

\begin{proof}
Without loss of generality we can assume \(X_0\) is affine,
geometrically reduced, geometrically irreducible, and has dimension
equal to \(d\).  After passing to a finite extension, we can find a
finite morphism \(X_0 \to \mathbf{A}^{d}_{\mathbf{F}_{q}}\).  Then
there is a constant \(C\) independent of \(e\), such that
\[
\# X_0(\mathbf{F}_{q^e}) \leq  C (q^{e})^{d}, \quad e = 1,2,\ldots.
\]
Let
\(r=\max\{\dim_{\overline{\mathbf{Q}}_{\ell}}\mathcal{F}_{\overline{x}}\}\)
where \(\overline{x}\) runs through all geometric points of \(X_{0}\).
Then for any closed point \(x\) of \(X_0\), we have
\[
|\iota \mathrm{Tr}(F_x, \mathcal{F}_{\overline{x}})| \leq r \cdot (q^{\deg x})^{\frac{w}{2}}.
\]
Therefore, the series
\[
\iota \left(\frac{L^{\prime}(\mathcal{F}_0,t)}{L(\mathcal{F}_0,t)}\right)
= \sum_{e\geq 1} \sum_{x \in X_0(\mathbf{F}_{q^e})} \iota \mathrm{Tr}(F_{x}, \mathcal{F}_{\overline{x}}) \cdot t^{e-1}
\]
has a majorant
\(rC \cdot \sum_{e\geq 1} (q^{d+\frac{w}{2}})^{e}\cdot |t|^{e-1}\).  Hence,
\(\iota \left( L^{\prime}(\mathcal{F}_0,t) / L(\mathcal{F}_0,t)\right)\)
is convergent on the disk \(|t| < q^{-w/2-d}\), and so is
\(\log L(\mathcal{F}_0,t)\).  The lemma then follows after taking
exponentiation.
\end{proof}

\begin{lemma}\label{lemma:simpler}
Let \(U_{0}\) be a smooth, geometrically connected, affine curve over
\(\mathbf{F}_{q}\).  Let \(\mathcal{E}_{0}\) be a lisse
\(\overline{\mathbf{Q}}_{\ell}\)-sheaf on \(U_{0}\) of pointwise
weight \(\leq w\).  Then for any reciprocal zero \(\alpha\) of
\(\det(1-tF, \mathrm{H}^{1}_{c}(U;\mathcal{E}))\), \(\alpha\) has
\(q\)-weight \(\leq w+2\).
\end{lemma}

\begin{proof}
By Grothendieck's trace formula, we have
\[
L(\mathcal{E}_{0},t)
= \frac{\det(1-tF, \mathrm{H}^{1}_{c}(U;\mathcal{E}))}
{\det(1-tF, \mathrm{H}^{2}_{c}(U;\mathcal{E}))}.
\]
If \(\alpha\) is not a reciprocal zero of
\(\det(1-tF, \mathrm{H}^{2}_{c}(U;\mathcal{E}))\), then \(\alpha\) is
a reciprocal zero of \(L(\mathcal{E}_{0},t)\).  In this case, the result
follows from Lemma~\ref{lemma:simple}.

It remains to consider the case in which \(\alpha\) is a reciprocal zero
of \(\det(1-tF, \mathrm{H}^{2}_{c}(U;\mathcal{E}))\).  Let \(x\) be a
closed point of \(U_{0}\).  Then
\[
\mathrm{H}^{2}_{c}(U;\mathcal{E}) \simeq
(\mathcal{E}_{\overline{x}}\otimes
\overline{\mathbf{Q}}_{\ell}(-1))_{\pi_{1}(U,\overline{x})},
\]
the monodromy coinvariant part of
\(\mathcal{E}_{\overline{x}}\otimes
\overline{\mathbf{Q}}_{\ell}(-1)\); see~\cite[(2.9)]{deligne_weil1}.
Thus \(\alpha^{\deg x}\) is an eigenvalue of \(F_{x}\) acting on
\(\mathcal{E}_{\overline{x}}\otimes
\overline{\mathbf{Q}}_{\ell}(-1)\).  By the pointwise weight hypothesis,
we have
\[
|\iota(\alpha^{\deg x})| \leq (q^{\deg x})^{1+\frac{w}{2}}.
\]
Taking roots yields the desired result.
\end{proof}

\begin{lemma}[Semicontinuity of weights]\label{lemma:semicont}
Let \(C_{0}\) be a nonsingular geometrically connected curve over
\(\mathbf{F}_{q}\).  Let
\(\mathcal{E}_{0}\) be a lisse sheaf on a Zariski open subset
\(j\colon U_{0} \hookrightarrow C_{0}\).  Suppose \(\mathcal{E}_{0}\) has
pointwise weight \(\leq w\).  Then \(\mathrm{R}^{0}j_{\ast}\mathcal{E}_{0}\)
has pointwise weight \(\leq w\) as well.
\end{lemma}

\begin{proof}[Proof \textup{(cf.~{\cite[Lemme~1.8.1]{deligne_weil2}, \cite[Lemma~I~2.5]{kiehl-weissauer_weil-conjectures-perverse-sheaves}})}]
Without loss of generality, we may assume that \(C_{0}\) is affine and
\(C_0 \setminus U_0\) consists of \(\mathbf{F}_q\)-rational
points.  In this case, we have
\(\mathrm{H}^{0}_{c}(C;\mathrm{R}^{0}j_{\ast}\mathcal{E})=0\).  This can be
seen, for example, from Artin's vanishing theorem
(Theorem~\ref{theorem:relative-affine-lefschetz}, with
\(S = \operatorname{Spec}\overline{\mathbf{F}}_q\)) and the perversity
of \(\mathrm{R}^{0}j_{\ast} \mathcal{E}[1]\).

Consider the exact sequence
\begin{equation}\label{eq:semic-0}
0 \to j_{!}\mathcal{E}_{0} \to \mathrm{R}^{0}j_{\ast}\mathcal{E}_{0} \to \mathcal{G}_{0} \to 0
\end{equation}
where \(\mathcal{G}_{0}\) is a sheaf supported on the finite set
\(C_{0}\setminus U_{0}\).  For any point \(x\) in
\(C_{0} \setminus U_{0}\), we have
\(\mathcal{G}_{0\overline{x}}=(\mathrm{R}^{0}j_{\ast}\mathcal{E}_{0})_{\overline{x}}\).
Applying the multiplicativity of \(L\)-functions to
\eqref{eq:semic-0}, we have
\begin{equation}\label{eq:semic-1}
L(\mathrm{R}^{0}j_{\ast}\mathcal{E}_{0},t)
= L(\mathcal{E}_{0},t) \cdot \prod_{x \in C_{0}\setminus U_{0}} \det(1 - tF_{x}, (\mathrm{R}^{0}j_{\ast}\mathcal{E}_{0})_{\overline{x}})^{-1}.
\end{equation}

On the other hand, taking compactly supported cohomology in
\eqref{eq:semic-0} and using that
\(\mathrm{H}^{1}_{c}(C;\mathcal{G}) =
\mathrm{H}^{2}_{c}(C;\mathcal{G})=0\), we find that
\begin{equation*}
\mathrm{H}^{2}_{c}(C;\mathrm{R}^{0}j_{\ast}\mathcal{E}) \simeq \mathrm{H}^{2}_{c}(U;\mathcal{E}).
\end{equation*}
Therefore, Grothendieck's trace formula gives
\begin{equation}
\label{eq:semic-2}
L(\mathrm{R}^{0}j_{\ast}\mathcal{E}_{0},t) = \frac{\det(1-tF, \mathrm{H}^{1}_{c}(C;\mathrm{R}^{0}j_{\ast}\mathcal{E}))}{\det(1-tF, \mathrm{H}^{2}_{c}(U;\mathcal{E}))}.
\end{equation}

Combining \eqref{eq:semic-1} and \eqref{eq:semic-2}, we get
\[
\frac{\det(1-tF, \mathrm{H}^{1}_{c}(U; \mathcal{E}))}{\det(1-tF, \mathrm{H}^{2}_{c}(U;\mathcal{E}))}
\cdot \prod_{x\in C_{0}\setminus U_{0}} \det(1-tF_{x}, (\mathrm{R}^{0}j_{\ast}\mathcal{E}_{0})_{\overline{x}})^{-1}
= \frac{\det(1-tF, \mathrm{H}^{1}_{c}(C;\mathrm{R}^{0}j_{\ast}\mathcal{E}))}{\det(1-tF, \mathrm{H}^{2}_{c}(U;\mathcal{E}))}.
\]
Clearing the common denominator
\(\det(1-tF, \mathrm{H}^{2}_{c}(U;\mathcal{E}))\),
we get
\[
\det(1-tF, \mathrm{H}^{1}_{c}(U;\mathcal{E})) =
\det(1-tF, \mathrm{H}^{1}_{c}(C;\mathrm{R}^{0}j_{\ast}\mathcal{E})) \cdot \prod_{x \in C_{0}\setminus U_{0}} \det(1-tF_{x}, (\mathrm{R}^{0}j_{\ast}\mathcal{E}_{0})_{\overline{x}}).
\]
In particular, for any \(x \in C_{0}\setminus U_{0}\), we have
\[
\det(1-tF_{x},(\mathrm{R}^{0}j_{\ast}\mathcal{E}_{0})_{\overline{x}})
\quad \text{divides} \quad
\det(1-tF, \mathrm{H}^{1}_{c}(U;\mathcal{E})).
\]
By
Lemma~\ref{lemma:simpler} above, the Frobenius eigenvalues
\(\alpha\) of \((\mathrm{R}^{0}j_{\ast}\mathcal{E}_{0})_{\overline{x}}\) must
satisfy \(|\iota \alpha| \leq q^{\frac{w+2}{2}}\).

We can repeat the same argument with \(\mathcal{E}_{0}\) replaced by
\(\mathcal{E}_{0}^{\otimes m}\).  Since
\(\mathrm{R}^{0}j_{\ast}(\mathcal{E}_{0}^{\otimes m})\) contains
\((\mathrm{R}^{0}j_{\ast} \mathcal{E}_{0})^{\otimes m}\) as a
subsheaf, and the Frobenius eigenvalues of the latter are the
\(m\)\textsuperscript{th} powers of those of
\(\mathrm{R}^{0}j_{\ast}\mathcal{E}_{0}\), we find that the Frobenius
eigenvalues \(\alpha\) of \(\mathrm{R}^{0}j_{\ast}\mathcal{E}_{0}\)
satisfy
\begin{equation*}
|\iota \alpha^{m}| \leq q^{\frac{wm+2}{2}}, \quad\text{or}\quad |\iota\alpha| \leq q^{\frac{w}{2}+\frac{1}{m}}.
\end{equation*}
Letting \(m \to \infty\) concludes the proof.
\end{proof}

\section{Proof of Theorem~\ref*{theorem:wrh}}
\label{sec:proof}

We proceed by induction on the dimension \(n\) of \(X_{0}\).  We also
remark that the theorem is insensitive to the choice of base field, so
we may replace \(\mathbf{F}_q\) by a finite extension as needed in the
argument.

\begin{situation}[Reduction to the case of an affine hypersurface]
\label{sec:reduction-to-hypersurface}
After replacing \(\mathbf{F}_q\) by a suitable finite extension, we
may assume that all the geometric irreducible components of \(X_{0}\)
are defined over \(\mathbf{F}_q\).  By induction on the number of
geometric irreducible components, we may assume that \(X_{0}\) is
itself geometrically irreducible of pure dimension \(n\).
Furthermore, we can find an open dense subscheme
\(U_{0} \subset X_{0}\) which is isomorphic to an open dense subscheme
of a hypersurface \(Y_{0}\) in affine space
\(\mathbf{A}^{n+1}_{\mathbf{F}_q}\).  Consider the exact sequence
\[
\mathrm{H}^{i}_{c}(U;\overline{\mathbf{Q}}_{\ell}) \to \mathrm{H}^{i}_c(X;\overline{\mathbf{Q}}_{\ell}) \to \mathrm{H}^{i}_{c}(X \setminus U;\overline{\mathbf{Q}}_{\ell}) ,
\]
and observe that the complement \(X \setminus U\) has dimension at
most \(n-1\).  By the induction hypothesis, it suffices to prove the
theorem for \(U_{0}\).  Let \(Z_{0} = Y_{0} \setminus U_{0}\);
clearly, \(\dim Z_{0} \leq n-1\).  The following exact sequence,
\[
\mathrm{H}^{i-1}_c(Z;\overline{\mathbf{Q}}_{\ell}) \to \mathrm{H}^{i}_{c}(U;\overline{\mathbf{Q}}_{\ell}) \to \mathrm{H}^{i}_{c}(Y;\overline{\mathbf{Q}}_{\ell}),
\]
shows that it is further sufficient to prove the theorem for the affine
hypersurface \(Y_0\).  Thus, after these reductions, we may assume
\(X_{0}\) is a hypersurface in \(\mathbf{A}^{n+1}_{\mathbf{F}_q}\),
defined by a single polynomial
\(f \in \mathbf{F}_q[T_1,\ldots,T_{n+1}]\).

If the projective closure
of \(X_0\) is smooth, then Theorem~\ref{theorem:wrh} follows from
Katz's theorem (Theorem~\ref{theorem:katz}) directly:

\begin{lemma}\label{lemma:hype}
Assume that Theorem~\ref{theorem:wrh} holds for all separated schemes
of finite type over \(\mathbf{F}_q\) of dimension at most \(n-1\).
Let \(X_0 \subset \mathbf{A}^{n+1}_{\mathbf{F}_q}\) be a hypersurface
such that its projective closure is a smooth hypersurface in
\(\mathbf{P}^{n+1}_{\mathbf{F}_q}\).  Then Theorem~\ref{theorem:wrh}
holds for \(X_0\).
\end{lemma}

\begin{proof}
Observe that the complement \(\overline{X}_0 \setminus X_0\) has
dimension at most \(n-1\).  By the inductive hypothesis and
Theorem~\ref{theorem:katz}, the exact sequence
\[
\mathrm{H}_{c}^{i-1}(\overline{X} \setminus X; \overline{\mathbf{Q}}_{\ell})
\to \mathrm{H}^{i}_c(X;\overline{\mathbf{Q}}_{\ell}) \to
\mathrm{H}^{i}(\overline{X};\overline{\mathbf{Q}}_{\ell})
\]
shows that Theorem~\ref{theorem:wrh} holds for \(X_0\).
\end{proof}

In our reduction of Theorem~\ref{theorem:wrh} to the case of affine
hypersurfaces, we have no control over the singularities of the
resulting hypersurface or of its projective closure.  As a result, to
complete the proof, we require an additional degeneration argument to
adequately address these potential singularities.
\end{situation}

\begin{situation}[Singular affine hypersurface: weight of middle cohomology]
\label{sec:middle-singular-hypersurface}
For a not necessarily irreducible \(n\)-dimensional hypersurface
\(X_{0} \subset \mathbf{A}_{\mathbf{F}_{q}}^{n+1}\) defined by a degree \(d\) polynomial
\(f \in \mathbf{F}_{q}[T_{1},\ldots,T_{n+1}]\), we now prove that the
Frobenius eigenvalues of
\(\mathrm{H}^{n}_{c}(X;\overline{\mathbf{Q}}_{\ell})\) are
of \(q\)-weight at most \(n\).

After possibly extending the base field, we can find a polynomial
\(g \in \mathbf{F}_q[T_1, \ldots, T_{n+1}]\) such that:
\begin{itemize}
\item \(\deg f = \deg g = d\),
\item the affine hypersurface \(W_0\) defined by \(g\) is smooth; and
\item the degree-\(d\) homogenization
\(T^d_{0} g(T_1 T_0^{-1},\ldots, T_{n+1}T_0^{-1})\) of \(g\) cuts out a
smooth projective hypersurface in \(\mathbf{P}^{n+1}_{\mathbf{F}_q}\).
\end{itemize}

Next, we introduce a family which interpolates between our original
hypersurface \(X_0\) and this auxiliary smooth hypersurface \(W_0\).
Consider
\[
\mathscr{X}_0 \subset \mathbf{A}^1_{\mathbf{F}_q} \times \mathbf{A}^{n+1}_{\mathbf{F}_q}
\]
defined by the equation
\[
t g(T_1, \ldots, T_{n+1}) + (1 - t) f(T_1, \ldots, T_{n+1}) = 0.
\]

Let \(\mathcal{H}_0\) denote the affine space parametrizing
polynomials in \(T_1, \dots, T_{n+1}\) of degree at most
\(d = \deg f\).  Within \(\mathcal{H}_0\), the locus \(\mathcal{U}_0\)
of polynomials whose degree-\(d\) homogenizations cut out smooth
hypersurfaces in \(\mathbf{P}^{n+1}\) is an open dense subscheme.  The
pencil \(t g + (1-t) f\) induces a morphism
\(u \colon \mathbf{A}^1_{\mathbf{F}_q} \to \mathcal{H}_0\).  Since
\(u(1) \in \mathcal{U}_0\), the inverse image
\(u^{-1}(\mathcal{U}_0)\) is a nonempty open subscheme of
\(\mathbf{A}^1_{\mathbf{F}_q}\), and hence dense.  Consequently, for
all but finitely many \(t \in \overline{\mathbf{F}}_q\), the
polynomial \(tg + (1-t)f\) defines a smooth hypersurface in
\(\mathbf{A}^{n+1}_{\overline{\mathbf{F}}_q}\), and its degree-\(d\)
homogenization cuts out a smooth hypersurface in
\(\mathbf{P}^{n+1}_{\overline{\mathbf{F}}_q}\).

Since
\(\mathbf{A}^1_{\mathbf{F}_q} \times \mathbf{A}^{n+1}_{\mathbf{F}_q}\)
is smooth, Lemma~\ref{lemma:perversity-local-system-on-smooth-variety}
shows that the shifted constant sheaf
\[
\overline{\mathbf{Q}}_{\ell,\mathbf{A}^1_{\mathbf{F}_q} \times
  \mathbf{A}^{n+1}_{\mathbf{F}_q}}[n+2]
\]
is perverse.  Hence, by
Lemma~\ref{lemma:local-equation-and-perverse-coconnective}, we see
that \(\overline{\mathbf{Q}}_{\ell,\mathscr{X}_0}[n+1]\) is perverse
coconnective, as \(\mathscr{X}_0\) is defined locally by a single
regular function.

Consider the projection
\(\pi\colon\mathscr{X}_0 \to \mathbf{A}^1_{\mathbf{F}_q}\).  This is an
affine morphism.  So, by Artin's vanishing theorem
(Theorem~\ref{theorem:relative-affine-lefschetz}), the complex
\(\mathrm{R}\pi_!(\overline{\mathbf{Q}}_{\ell,\mathscr{X}_0}[n+1])\) is
perverse coconnective on \(\mathbf{A}^1_{\mathbf{F}_q}\).

Choose an open embedding
\(j\colon V_0 \hookrightarrow \mathbf{A}^1_{\mathbf{F}_q}\)
of a dense Zariski open subset such that
\begin{itemize}
\item for all \(m\),
the cohomology sheaves \(\mathrm{R}^m \pi_!(\overline{\mathbf{Q}}_{\ell,\mathscr{X}_0})\) of
\(\mathrm{R}\pi_!(\overline{\mathbf{Q}}_{\ell,\mathscr{X}_0})\) are lisse on
\(V_0\), and
\item for every geometric point \(t \in V\), the projective
hypersurface defined by the degree-\(d\) homogenization of
\(tg + (1-t)f\) is smooth.
\end{itemize}

By Lemma~\ref{lemma:hype} and base change for cohomology with
compact support, the sheaf
\((\mathrm{R}^n\pi_!\overline{\mathbf{Q}}_{\ell,\mathscr{X}_0})|_{V_0}\) is of
pointwise weight \(\leq n\), that is, for
all closed points \(x\) of \(V_0\), the eigenvalues of Frobenius
\(F_x\) acting on the stalks are of \(\#\kappa(x)\)-weight at most
\(n\).

By the semicontinuity of weights
(Lemma~\ref{lemma:semicont}) applied to
the inclusion \(j\colon V_0 \to \mathbf{A}^1_{\mathbf{F}_q}\) and the lisse sheaf \((\mathrm{R}^n \pi_{!}\overline{\mathbf{Q}}_{\ell,\mathscr{X}_0})|_{V_0}\),
the stalk at \(0\) of the sheaf
\(\mathrm{R}^0 j_* ((\mathrm{R}^n \pi_! \overline{\mathbf{Q}}_{\ell,\mathscr{X}_0})|_{V_0})\) has \(q\)-weight
\(\leq n\).
Because \(\mathrm{R}\pi_{!}(\overline{\mathbf{Q}}_{\ell,\mathscr{X}_{0}}[n+1])\) is perverse coconnective,
and since \(\mathrm{R}^n \pi_{!} \overline{\mathbf{Q}}_{\ell,\mathscr{X}_{0}} = \mathcal{H}^{-1}(\mathrm{R}\pi_{!}(\overline{\mathbf{Q}}_{\ell,\mathscr{X}_{0}}[n+1]))\),
we can apply the perverse degeneration lemma
(Lemma~\ref{lemma:perverse-degeneration-lemma}) to \(\mathcal{F} = \mathrm{R}\pi_{!}(\overline{\mathbf{Q}}_{\ell,\mathscr{X}_0}[n+1])\), and deduce that the
compactly supported cohomology group
\(\mathrm{H}^n_c(X;\overline{\mathbf{Q}}_{\ell})\) also has
\(q\)-weight at most \(n\).


\end{situation}

\begin{situation}[Singular affine hypersurface: other cohomology degrees]
\label{sec:beyond-middle}
Finally, we address the remaining cohomology degrees.  Since \(X\) is
defined by one polynomial in \(\mathbf{A}^{n+1}\),
\(\overline{\mathbf{Q}}_{\ell,X}[n]\) is a perverse sheaf on \(X\) by
Lemma~\ref{lemma:perversity-local-system-on-smooth-variety} and
Lemma~\ref{lemma:local-equation-and-perverse-coconnective}.  We need
to show that for each integer \(i \neq 0\), the eigenvalues of
Frobenius acting on
\(\mathrm{H}^{n+i}_c(X;\overline{\mathbf{Q}}_{\ell})\) are of
\(q\)-weight at most \(n+i\).  If \(i \leq -1\), Artin's vanishing
theorem guarantees that
\(\mathrm{H}^{n+i}_c(X;\overline{\mathbf{Q}}_{\ell}) = 0\), so there
is nothing to prove.  For \(i \geq 1\), the weak Lefschetz theorem
(Lemma~\ref{lemma:gysin}) provides a surjection
\[
\mathrm{H}^{n-i}_c(X \cap L^{(i)};\overline{\mathbf{Q}}_{\ell}(-i))
\to \mathrm{H}^{n+i}_c(X;\overline{\mathbf{Q}}_{\ell}),
\]
where \(L_0^{(i)}\) is a sufficiently general codimension \(i\) linear
subspace of \(\mathbf{A}^{n+1}_{\mathbf{F}_q}\) (again, we may pass to
a finite field extension in order to find such a subspace).  Since
\(L^{(i)}\) is chosen in a general fashion,
we have \(\dim X \cap L^{(i)} = n-i\),
and \(X \cap L^{(i)}\)
is a hypersurface in an \((n-i+1)\)-dimensional affine space.
By the previous step, the
Frobenius eigenvalues of the group
\(\mathrm{H}^{n-i}_c(X \cap L^{(i)};\overline{\mathbf{Q}}_{\ell})\)
are of \(q\)-weight at most \(n-i\); so taking into account the Tate
twist, the Frobenius eigenvalues of
\(\mathrm{H}^{n+i}_c(X;\overline{\mathbf{Q}}_{\ell})\) are of
\(q\)-weight at most \(n+i\).  This completes the proof.\qed
\end{situation}

\bibliographystyle{plain}
\bibliography{rh_revised}

\begin{thebibliography}{10}

\bibitem{sga4}
{\em Th\'eorie des topos et cohomologie \'etale des sch\'emas.}, volume 269,
  270, 305 of {\em Lecture Notes in Mathematics}.
\newblock Springer-Verlag, Berlin-New York, 1972--73.
\newblock S\'eminaire de G\'eom\'etrie Alg\'ebrique du Bois-Marie 1963--1964
  (SGA 4), Dirig\'e{} par M. Artin, A. Grothendieck, et J. L. Verdier. Avec la
  collaboration de N. Bourbaki, P. Deligne et B. Saint-Donat.

\bibitem{bbd}
Alexander~A. Beilinson, Joseph Bernstein, Pierre Deligne, and Ofer Gabber.
\newblock Faisceaux pervers.
\newblock In {\em Analysis and topology on singular spaces, {I} ({L}uminy,
  1981)}, volume 100 of {\em Ast\'erisque}, pages 5--171. Soc. Math. France,
  Paris, 1982.

\bibitem{deligne_weil1}
Pierre Deligne.
\newblock La conjecture de {W}eil. {I}.
\newblock {\em Inst. Hautes \'Etudes Sci. Publ. Math.}, (43):273--307, 1974.

\bibitem{deligne_finitude}
Pierre Deligne.
\newblock Th\'eor\`emes de finitude en cohomologie {$\ell$}-adique.
\newblock In {\em Cohomologie \'etale}, volume 569 of {\em Lecture Notes in
  Math.}, pages 233--261. Springer, Berlin, 1977.

\bibitem{deligne_weil2}
Pierre Deligne.
\newblock La conjecture de {W}eil. {II}.
\newblock {\em Inst. Hautes \'Etudes Sci. Publ. Math.}, (52):137--252, 1980.

\bibitem{katz_affine-cohomological-transforms}
Nicholas~M. Katz.
\newblock Affine cohomological transforms, perversity, and monodromy.
\newblock {\em J. Amer. Math. Soc.}, 6(1):149--222, 1993.

\bibitem{katz_perverse-origin}
Nicholas~M. Katz.
\newblock A semicontinuity result for monodromy under degeneration.
\newblock {\em Forum Math.}, 15(2):191--200, 2003.

\bibitem{katz_riemann-hypothesis-for-curves-and-hypersurfaces}
Nicholas~M. Katz.
\newblock A note on {R}iemann hypothesis for curves and hypersurfaces over
  finite fields.
\newblock {\em Int. Math. Res. Not. IMRN}, (9):2328--2341, 2015.

\bibitem{kiehl-weissauer_weil-conjectures-perverse-sheaves}
Reinhardt Kiehl and Rainer Weissauer.
\newblock {\em Weil conjectures, perverse sheaves and {$l$}-adic {F}ourier
  transform}, volume~42 of {\em Ergebnisse der Mathematik und ihrer
  Grenzgebiete. 3. Folge. A Series of Modern Surveys in Mathematics}.
\newblock Springer-Verlag, Berlin, 2001.

\bibitem{scholl_hypersurfaces-and-weil-conjectures}
Anthony~J. Scholl.
\newblock Hypersurfaces and the {W}eil conjectures.
\newblock {\em Int. Math. Res. Not. IMRN}, (5):1010--1022, 2011.

\bibitem{wan-zhang_colevel}
Daqing Wan and Dingxin Zhang.
\newblock Hodge and {F}robenius colevels of algebraic varieties.
\newblock {\em J. Lond. Math. Soc. (2)}, 111(4):Paper No. e70160, 26 pp, 2025.

\bibitem{wan-zhang_betti-number-bounds-for-varieties-and-exponential-sums0}
Daqing Wan and Dingxin Zhang.
\newblock Betti number bounds for varieties and exponential sums.
\newblock {\em Adv. Math.}, 490:Paper No. 110852, 2026.

\bibitem{weil_equations-in-finite-fields}
Andr\'e Weil.
\newblock Numbers of solutions of equations in finite fields.
\newblock {\em Bull. Amer. Math. Soc.}, 55:497--508, 1949.

\end{thebibliography}

\end{document}